\documentclass[12pt,dvips]{amsart}
\usepackage{amsfonts, amssymb, latexsym, epsfig, enumerate}

\newtheorem{Theorem}{Theorem}[section]
\newtheorem{Proposition}[Theorem]{Proposition}
\newtheorem{Lemma}[Theorem]{Lemma}
\newtheorem{Claim}[Theorem]{Claim}

\newtheorem{Corollary}[Theorem]{Corollary}

\newtheorem{Definition-Proposition}[Theorem]{Definition-Theorem}
\newtheorem{Main Conjecture}[Theorem]{Main Conjecture}

\newtheorem{Remark}[Theorem]{Remark}
\theoremstyle{remark}
\newtheorem{Example}[Theorem]{Example}

\renewcommand{\theenumi}{\roman{enumi}}

\newcommand{\excise}[1]{}

\newcommand{\itop}{\tau}  % index set to partition operator
\newcommand{\Gr}{\mathrm{Gr}}
\newcommand{\CC}{\mathbb{C}}
\newcommand{\ZZ}{\mathbb{Z}}
\newcommand{\define}{\textbf}

\newcommand{\eqssyt}{\mathtt{EqSSYT}}
\newcommand{\dist}{\mathrm{dist}}
\newcommand{\xbox}{\mathtt{x}}
\newcommand{\factor}{\mathtt{apfactor}}
\newcommand{\wt}{\mathtt{apwt}}

\theoremstyle{plain}

\newcommand{\cellsize}{19}
\newlength{\cellsz} 
\newsavebox{\cell}
\sbox{\cell}{\begin{picture}(\cellsize,\cellsize)
\put(0,0){\line(1,0){\cellsize}}
\put(0,0){\line(0,1){\cellsize}}
\put(\cellsize,0){\line(0,1){\cellsize}}
\put(0,\cellsize){\line(1,0){\cellsize}}
\end{picture}}
\newcommand\cellify[1]{\def\thearg{#1}\def\nothing{}%
\ifx\thearg\nothing
\vrule width0pt height\cellsz depth0pt\else
\hbox to 0pt{\usebox{\cell} \hss}\fi%
\vbox to \cellsz{
\vss
\hbox to \cellsz{\hss$#1$\hss}
\vss}}
\newcommand\tableau[1]{\vtop{\let\\\cr
\baselineskip -16000pt \lineskiplimit 16000pt \lineskip 0pt
\ialign{&\cellify{##}\cr#1\crcr}}}
\newcommand{\kellsize}{12}
\newlength{\kellsz} 
\newsavebox{\kell}
\sbox{\kell}{\begin{picture}(\kellsize,\kellsize)
\put(0,0){\line(1,0){\kellsize}}
\put(0,0){\line(0,1){\kellsize}}
\put(\kellsize,0){\line(0,1){\kellsize}}
\put(0,\kellsize){\line(1,0){\kellsize}}
\end{picture}}
\newcommand\kellify[1]{\def\thearg{#1}\def\nothing{}%
\ifx\thearg\nothing
\vrule width0pt height\kellsz depth0pt\else
\hbox to 0pt{\usebox{\kell} \hss}\fi%
\vbox to \kellsz{
\vss
\hbox to \kellsz{\hss$#1$\hss}
\vss}}
\newcommand\ktableau[1]{\vtop{\let\\\cr
\baselineskip -16000pt \lineskiplimit 16000pt \lineskip 0pt
\ialign{&\kellify{##}\cr#1\crcr}}}

\newcommand{\sellsize}{36}
\newlength{\sellsz} 
\newsavebox{\sell}
\sbox{\sell}{\begin{picture}(\sellsize,20)
\put(0,0){\line(1,0){\sellsize}}
\put(0,0){\line(0,1){\sellsize}}
\put(\sellsize,0){\line(0,1){\sellsize}}
\put(0,\sellsize){\line(1,0){\sellsize}}
\end{picture}}
\newcommand\sellify[1]{\def\thearg{#1}\def\nothing{}%
\ifx\thearg\nothing
\vrule width0pt height\sellsz depth0pt\else
\hbox to 0pt{\usebox{\sell} \hss}\fi%
\vbox to \sellsz{
\vss
\hbox to \sellsz{\hss$#1$\hss}
\vss}}
\newcommand\stableau[1]{\vtop{\let\\\cr
\baselineskip -16000pt \lineskiplimit 16000pt \lineskip 0pt
\ialign{&\sellify{##}\cr#1\crcr}}}
\begin{document}
\pagestyle{plain}

\mbox{}
\title[Eigenvalues and equivariant cohomology]{Eigenvalues of Hermitian matrices and\\ equivariant cohomology of Grassmannians}

\author{David Anderson}
\address{FSMP--Institut de Math\'ematiques de Jussieu, 75013
Paris, France}
\email{andersond@math.jussieu.fr}

\author{Edward Richmond}
\address{Department of Mathematics, University of British Columbia,
Vancouver, BC, V6T 1Z2, Canada }
\email{erichmond@math.ubc.ca}

\author{Alexander Yong}
\address{Department of Mathematics, University of Illinois at
Urbana-Champaign, Urbana, IL 61801, USA}

\email{ayong@uiuc.edu}

\date{April 2, 2013}

\begin{abstract}
The saturation theorem of [Knutson-Tao '99]
concerns the nonvanishing of
 Littlewood-Richardson coefficients. In combination with
 work of [Klyachko '98], it implies [Horn '62]'s conjecture about eigenvalues
 of sums of Hermitian matrices.
   This eigenvalue problem has a generalization [Friedland '00] to \emph{majorized} sums of Hermitian matrices.

  We further illustrate the common features between these two eigenvalue problems and their connection to Schubert calculus of Grassmannians. Our main result gives a Schubert
  calculus interpretation of Friedland's problem, via \emph{equivariant} cohomology of Grassmannians.
  In particular, we prove a saturation theorem for this setting.
  Our arguments employ the aformentioned work together with [Thomas-Yong '12].
\end{abstract}

\maketitle
%%%%%%%%%%%%%%%%%%%%%%%%%%%%%%%%%%%%%%%
\section{Introduction and the main results}
%%%%%%%%%%%%%%%%%%%%%%%%%%%%%%%%%%%%%%%

\subsection{Eigenvalue problems of  A.~Horn and of S.~Friedland}
The eigenvalue problem for Hermitian matrices asks how imposing the
condition $A+B=C$ on three $r\times r$ Hermitian matrices constrains their eigenvalues $\lambda$,
$\mu$, and $\nu$, written as weakly decreasing vectors of real numbers.
This problem was considered in the 19th century, and has reappeared in various guises since.  A general survey is given in \cite{Fultona}; here we mention a few highlights of the story.  Building on observations of H.~Weyl, K.~Fan, and others, A.~Horn recursively defined a list of inequalities on triples $(\lambda,\mu,\nu) \in {\mathbb R}^{3r}$, and conjectured that these give a complete solution to the eigenvalue problem \cite{Horn}.  The fact that these inequalities (or others that turn out to be equivalent) are necessary
has been proved by several authors, including B.~Totaro \cite{Totaro} and A.~Klyachko \cite{Klyachko}.  Klyachko also established that his list of inequalities is sufficient, giving the first solution to the eigenvalue problem.

In fact, he showed more: the same inequalities give an asymptotic solution to the problem of which Littlewood-Richardson coefficients $c_{\lambda,\mu}^{\nu}$ are nonzero.  More precisely, suppose $\lambda,\mu,\nu$ are partitions with at most $r$ parts.  Klyachko showed that if $c_{\lambda,\mu}^{\nu}\neq 0$, then
$(\lambda,\mu,\nu)\in {\mathbb Z}_{\geq 0}^{3r}$ satisfies his inequalities; conversely, if
$(\lambda,\mu,\nu)\in {\mathbb Z}_{\geq 0}^{3r}$ satisfy his inequalities then
$c_{N\lambda,N\mu}^{N\nu}\neq 0$ for some $N\in {\mathbb N}$.  (Here, $N\lambda$ is the partition with each part of $\lambda$ stretched by a factor of $N$.)
Sharpening this last statement, A.~Knutson-T.~Tao \cite{Knutson.Tao:99} established the
\emph{saturation theorem}: $c_{\lambda,\mu}^{\nu}\neq 0$ if and only if
$c_{N\lambda,N\mu}^{N\nu}\neq 0$.  Combined with \cite{Klyachko}, it follows that Klyachko's
solution agrees with Horn's conjectured solution.

The Littlewood-Richardson coefficients are structure constants for multiplication of Schur polynomials.
Therefore, they can be alternatively interpreted as tensor product
multiplicities in the representation theory of $GL_n$, or as intersection multiplicities in the Schubert calculus of Grassmannians.  Indeed, \cite{Knutson.Tao:99} adopts the former viewpoint,
providing conjectural extensions to other Lie groups. Subsequent work includes \cite{Kapovich.Millson, BK:10, Kumar:ICM, Ressayre, Sam}; see also the references therein.

The main goal of this paper is to provide further evidence of the naturality of the connection of Horn's problem to Schubert calculus.  We demonstrate how the connection persists for the following extension of this eigenvalue problem.  Recall that a Hermitian matrix $M$ \define{majorizes} another Hermitian matrix $M'$ if $M-M'$ is positive semidefinite (its eigenvalues are all nonnegative).  In this case, we write $M\geq M'$.  S.~Friedland \cite{Friedland} considered the following question:
\begin{equation}\nonumber
\mbox{\emph{Which eigenvalues $(\lambda,\mu,\nu)$ can occur if $A+B \geq C$?}}
\end{equation}
His solution is in terms of linear inequalities, which includes Klyachko's inequalities, a trace
inequality and some additional inequalities.  Later, W.~Fulton \cite{Fulton} proved the additional
inequalities are unnecessary.  See followup work by A.~Buch \cite{Buch:06} and by C.~Chindris
\cite{Chindris} (who extends the work of H.~Derksen-J.~Weyman \cite{Derksen.Weyman}).

Our finding is that the solution to S.~Friedland's problem also
governs the \emph{equivariant} Schubert calculus of Grassmannians.  This
parallels the Horn problem's connection to classical Schubert calculus, but separates
the problem from $GL_n$-representation theory.

Let $C_{\lambda,\mu}^{\nu}$ be the equivariant Schubert structure coefficient (defined in Section~1.2).
The analogy with the earlier results is illustrated by:

\begin{Theorem}[Equivariant saturation]
\label{claim:main}
$C_{\lambda,\mu}^{\nu}\neq 0$ if and only if $C_{N\cdot\lambda,N\cdot\mu}^{N\cdot \nu}\neq 0$ for any $N\in {\mathbb N}$.
\end{Theorem}

When $|\lambda|+|\mu|=|\nu|$ then $C_{\lambda,\mu}^{\nu}=c_{\lambda,\mu}^{\nu}$. Hence
Theorem~\ref{claim:main} actually generalizes the saturation theorem.  That said, our proofs rely
on the classical Horn inequalities and so do not provide an independent proof of the earlier
results.  In addition, we use the recent combinatorial rule for
$C_{\lambda,\mu}^{\nu}$ developed by H.~Thomas and the third author
\cite{Thomas.Yong}.\footnote{The easy direction of (equivariant) saturation,
$C_{\lambda,\mu}^{\nu}\neq 0\Rightarrow C_{N\cdot\lambda,N\cdot\mu}^{N\cdot \nu}\neq 0$, can be
proved directly by using this rule (or others). However, as in the classical situation, it is the
converse that is nonobvious.}

\noindent

\subsection{Equivariant cohomology of Grassmannians}\label{s.eqcoh}

Let $\Gr_r(\CC^n)$ denote the Grassmannian of $r$-dimensional subspaces $V \subseteq \CC^n$. This
space comes with an action of the torus $T=(\CC^*)^n$ (induced from the action of $T$ on $\CC^n$). Therefore, it makes sense
to discuss the \emph{equivariant cohomology ring} $H_T^*\Gr_r(\CC^n)$. This ring is an
algebra over $\ZZ[t_1,\ldots,t_n]$. (A more complete exposition of equivariant
cohomology may be found in, e.g., \cite{Fultoneq}.)

As a $\ZZ[t_1,\ldots,t_n]$-module, $H_T^*\Gr_r(\CC^n)$ has a basis of \emph{Schubert classes}.  To
define these, fix the flag of subspaces
\[F_\bullet: 0\subset F_1 \subset F_2 \subset \cdots \subset
F_n=\CC^n,\]
where $F_i$ is the span of the standard basis vectors
$e_n,\,e_{n-1},\,\ldots,e_{n+1-i}$.  For each Young diagram $\lambda$ inside the $r\times (n-r)$
rectangle, which we denote by $\Lambda$, there is a corresponding {\bf Schubert variety}, defined by
\[
  X_\lambda := \{ V \subseteq \CC^n \,|\, \dim(V \cap F_{n-r+i-\lambda_i}) \geq i, \text{ for } 1\leq i \leq r \}.
\]
Since $X_{\lambda}$ is invariant under the action of $T$, and has codimension $2|\lambda|$, it determines a class $[X_\lambda]$ in $H_T^{2|\lambda|}\Gr_r(\CC^n)$.  As $\lambda$ varies over all
Young diagrams inside $\Lambda$, the classes $[X_\lambda]$ form a basis for $H_T^{*}\Gr_r(\CC^n)$
over $\ZZ[t_1,\ldots,t_n]$. Therefore in $H_T^{*}\Gr_r(\CC^n)$ we have
\[
  [X_\lambda]\cdot [X_\mu] = \sum_{\nu \subseteq \Lambda} C_{\lambda,\mu}^\nu [X_\nu],
\]
where the coefficients $C_{\lambda,\mu}^\nu \in \ZZ[t_1,\ldots,t_n]$ are the {\bf equivariant
Schubert structure coefficients}.  By homogeneity, $C_{\lambda,\mu}^\nu$ is a polynomial of degree
$|\lambda|+|\mu|-|\nu|$.  In particular, this coefficient is zero unless $|\lambda|+|\mu|\geq |\nu|$.

The polynomials $C_{\lambda,\mu}^{\nu}$ depend on the parameters $r$ and $n$, but our notation
drops this dependency, with the following justification. First, we already fixed $r$. Next, the  standard embedding $\iota\colon \Gr_r(\CC^n) \hookrightarrow \Gr_r(\CC^{n+1})$ induces a map $\iota^*\colon H_T^*\Gr_r(\CC^{n+1}) \to H_T^*\Gr_r(\CC^{n})$.  Using superscripts to indicate where a subvariety lives, we have $\iota^{-1}X_\lambda^{(n+1)} = X_\lambda^{(n)}$, and therefore $\iota^*[X_\lambda^{(n+1)}] = [X_\lambda^{(n)}]$.  Let us write $H_T^*$ for the graded inverse limit of these equivariant cohomology rings, so it is an algebra over $\ZZ[t_1,t_2,\ldots]$.  Write $\hat\sigma_\lambda \in H_T^*$ for the stable limit of the Schubert classes $[X_\lambda^{(n)}]$.  The same structure constants $C_{\lambda,\mu}^\nu$ describe
${\hat\sigma}_{\lambda}\cdot {\hat\sigma}_{\mu}$ in this
limit, so we can work in that limit without reference to $n$.

\subsection{Inequalities for $C_{\lambda,\mu}^{\nu}$ and for eigenvalues}

We will deduce Theorem~\ref{claim:main} from inequalities describing the nonvanishing of
$C_{\lambda,\mu}^{\nu}$. Let $[r]:=\{1,2,\ldots r\}$.  For any
\[I=\{i_1<i_2< \cdots<i_d\}\subseteq [r]\]
define the partition
\[
 \itop(I):=(i_d-d\geq \cdots\geq i_2-2 \geq i_1-1).
\]
This defines a bijection between subsets of $[r]$ of cardinality $d$ and partitions whose Young
diagrams are contained in a $d\times(r-d)$ rectangle. The following theorem combines the main
results of \cite{Klyachko, Knutson.Tao:99}.

\begin{Theorem}\label{thm:classicalHorn}(\cite{Klyachko}, \cite{Knutson.Tao:99})
Let $\lambda,\mu,\nu$ be partitions with at most $r$ parts such that
\begin{equation}
\label{eqn:classicalcond}
|\lambda|+|\mu|=|\nu|.
\end{equation}
The following are equivalent:
\begin{enumerate}
\item $c_{\lambda,\mu}^{\nu}\neq 0$.
\item For every $d<r$, and every triple of subsets $I,J,K\subseteq [r]$ of cardinality $d$ such that $c_{\itop(I),
\itop(J)}^{\itop(K)}\neq 0$, we have
\begin{equation}\label{eq:ineq}\sum_{i\in
I}\lambda_i+\sum_{j\in J}\mu_j\geq \sum_{k\in K}\nu_k.\end{equation}
\item There exist $r\times r$ Hermitian matrices $A,B,C$ with eigenvalues $\lambda,\mu,\nu$ such that $A+B=C.$
\end{enumerate}
\end{Theorem}

We are now ready to state our main result, which is a generalization of
Theorem~\ref{thm:classicalHorn}.

\begin{Theorem}\label{thm:equivHorn}Let $\lambda,\mu,\nu$ be partitions with at most $r$ parts such that
\begin{equation}
\label{eqn:EqHorncond}
|\lambda|+|\mu|\geq |\nu| \mbox{\ and  \ }\max\{\lambda_i,\mu_i\}\leq \nu_i \mbox{\ for all $i\leq r$.}
\end{equation}
The following are equivalent:
\begin{enumerate}
\item $C_{\lambda,\mu}^{\nu}\neq 0$.
\item For every $d<r$, and every triple of subsets $I,J,K\subseteq [r]$ of cardinality $d$ such that $c_{\itop(I),
\itop(J)}^{\itop(K)}\neq 0$, we have $$\sum_{i\in I}\lambda_i+\sum_{j\in J}\mu_j\geq \sum_{k\in
K}\nu_k.$$

\item There exist $r\times r$ Hermitian matrices $A,B,C$ with eigenvalues $\lambda,\mu,\nu$ such that $A+B\geq C$.
\end{enumerate}
\end{Theorem}

Theorem \ref{thm:equivHorn} asserts that the main recursive inequalities (ii) controlling nonvanishing of $C_{\lambda,\mu}^{\nu}$ are just Horn's inequalities  (\ref{eq:ineq}).  The only
difference between the governing inequalities lies in (\ref{eqn:classicalcond}) versus
(\ref{eqn:EqHorncond}).  Notice that the second condition in (\ref{eqn:EqHorncond}) is
unnecessary in Theorem~\ref{thm:classicalHorn} since it is already implied by
(\ref{eqn:classicalcond}) combined with \eqref{eq:ineq}.

In fact, we will use Theorem~\ref{thm:classicalHorn} to prove Theorem~\ref{thm:equivHorn}.  Moreover, the equivalence of conditions (ii) and (iii) is immediate from \cite{Friedland, Fulton}.
Since the inequalities of Theorem~\ref{thm:equivHorn} are homogeneous, the equivalence of (i) and (ii) immediately implies Theorem~\ref{claim:main}.

\subsection{Further comparisons to the literature}
The proof of the saturation theorem given in \cite{Knutson.Tao:99} is combinatorial, employing
their \emph{honeycomb model} for $c_{\lambda,\mu}^{\nu}$.  
In contrast, P.~Belkale first geometrically proves the equivalence ``(i)$\Leftrightarrow$(ii)'' of Theorem~\ref{thm:classicalHorn}, and then deduces the saturation theorem as an easy consequence \cite{Belkale:06}.
By comparison, our main tool is again a new combinatorial model for
$C_{\lambda,\mu}^{\nu}$; we similarly deduce (equivariant) saturation from the eigenvalue inequalities.

It seems plausible to give geometric proofs of our theorems, along the lines
of \cite{Belkale:06}, using the equivariant moving lemma of the first author \cite{Anderson}.  This approach is
especially pertinent where one does not have good combinatorial control of the equivariant Schubert coefficients, e.g., in the case of minuscule $G/P$, cf. \cite{Purbhoo.Sottile}.  (P.~Belkale and S.~Kumar \cite{BK:06} also consider the vanishing problem for classical Schubert structure constants associated to more general $G/P$'s.)

Schubert calculus on Grassmannians has two other basic extensions that have been extensively studied:
quantum and $K$-theoretic Schubert calculus.  It is therefore natural to ask if and how the Horn problem may extend in each of these directions.

The first of these was studied by P.~Belkale \cite{Belkale:08}, who established a
relationship between an eigenvalue problem for products of unitary matrices and analogues of the
saturation and Horn theorems for quantum cohomology of Grassmannians.  
(A combined quantum-equivariant extension is plausible, and investigating this seems worthwhile, but we have not yet undertaken such an investigation.)

For the second, let $k_{\lambda,\mu}^\nu$ denote the $K$-theoretic structure constant with respect to the basis
structure sheaves of Schubert varieties.  The ``easy'' implication
$k_{\lambda,\mu}^{\nu}\neq
0\implies k_{N\lambda,N\mu}^{N\nu}\neq 0$
is false in general. (However, it is not known if the
converse is true.)  For example, \cite[Section~7]{Buch:02} notes that
$k_{(1),(1)}^{(2,1)}=-1 \mbox{\ but  \ }  k_{(2),(2)}^{(4,2)}=0$.
One also can check that the same partitions give
a counterexample for saturation in $T$-equivariant
$K$-theory, as well. Moreover, consider structure constants ${\widetilde k}_{\lambda,\mu}^{\nu}$ for the multiplication of
the dual basis in $K$-theory. Using the rule of \cite[Theorem~1.6]{Thomas.Yong:direct}, one checks that ${\widetilde k}_{N\cdot(1),N\cdot(1)}^{N\cdot(2,1)}$ is nonzero
for $N=1,2,3$ but zero for $N=4$.

Summarizing, this paper addresses the remaining basic extension of Schubert calculus where a complete analogue of
the saturation theorem exists. The result linking Friedland's problem to equivariant Schubert calculus gives further evidence
towards the thesis that Schubert calculus is a natural perspective for Horn's problem.  That said, some room for clarification of this thesis remains: on one hand, the polynomials $C_{\lambda,\mu}^{\nu}$ also have representation theoretic interpretations \cite{Molev.Sagan:99}; on the other hand, saturation fails in $K$-theoretic Schubert calculus (in three forms).  Finding deeper connections and explanations for these phenomena seems an interesting possibility for future work.

\subsection{Organization}
In Section~2, we give the proof of Theorem~\ref{thm:equivHorn}, assuming a fact about the
equivariant coefficients (Proposition~\ref{prop:technical}) that we use in our inductive proof.  This in
turn is proved in Section~3, after a review of the combinatorial rule of \cite{Thomas.Yong}.

%%%%%%%%%%%%%%%%%%%%%%%%%%%%%%%%%%%%%%%%%%%%%%%%
\section{Proof of Theorem~\ref{thm:equivHorn}}
%%%%%%%%%%%%%%%%%%%%%%%%%%%%%%%%%%%%%%%%%%%%%%%%

As remarked above, we only need to show that parts (i) and (ii) of Theorem~\ref{thm:equivHorn} are equivalent.  We run an induction on the degree $p=|\lambda|+|\mu|-|\nu|$,
simultaneously with an induction on $r$.  
To do this, we identify two key nonvanishing criteria in the following proposition.
\begin{Proposition}
\label{prop:technical}
Assume $C_{\lambda,\mu}^{\nu}\neq 0$. 
\renewcommand{\theenumi}{\Alph{enumi}}
\begin{enumerate}
\item $C_{\lambda,\mu^{\uparrow}}^{\nu}\neq 0$ for any $\mu\subset \mu^{\uparrow}\subseteq \nu$. \label{propA}

\smallskip

\item If $|\nu|<|\lambda|+|\mu|$, then for any $s$ such that $|\nu|-|\lambda|\leq s < |\mu|$, there is a $\mu^{\downarrow}\subset \mu$ with $|\mu^{\downarrow}|=s$ and $C_{\lambda,\mu^{\downarrow}}^{\nu}\neq 0$.  (In particular, taking $s=|\nu|-|\lambda|$, we have $c_{\lambda,\mu^{\downarrow}}^\nu \neq 0$.)\label{propB}
\end{enumerate}
\renewcommand{\theenumi}{\roman{enumi}}
\end{Proposition}
We postpone the proof to the next section.

\begin{proof}[Proof of Theorem~\ref{thm:equivHorn}, (i) $\Rightarrow$ (ii)]
If $C_{\lambda,\mu}^{\nu}\neq 0,$ then by Proposition~\ref{prop:technical}(\ref{propB}), we can find $\lambda^{\downarrow}\subseteq \lambda$ such that
$|\lambda^{\downarrow}|+|\mu|=|\nu|$ and
$c_{\lambda^{\downarrow},\mu}^{\nu}\neq 0$.  
By Theorem~\ref{thm:classicalHorn}, for any triple $(I,J,K)$ such that $c_{\itop(I),\itop(J)}^{\itop(K)}\neq 0$, we have
\[
  \sum_{i\in I} \lambda^\downarrow_i + \sum_{i\in J} \mu_j \geq \sum_{k\in K} \nu_k.
\]
Since
$\sum_{i\in I} \lambda_i \geq \sum_{i\in I} \lambda^\downarrow_i$,
(i) implies (ii), as desired.
\end{proof}

Recall the bijection between $d$-subsets $I\subseteq [r]$ and partitions
$\lambda=\itop(I)$ in the $d\times(r-d)$ rectangle, and write $\sigma_I=\sigma_{\itop(I)}$ for
the corresponding Schubert class in the \emph{ordinary} cohomology ring $H^*\Gr_d(\CC^r).$  Define
$I^{\vee}\subseteq [r]$ as
\[I^{\vee}:=\{r+1-i_d<\cdots <r+1-i_2<r+1-i_1 \};\]
this is the subset associated to the shape $\lambda^{\vee}$ which is defined by taking the
complement of $\lambda$ in $d\times (r-d)$ and rotating by $180$ degrees.

We need an alternative characterization of Theorem \ref{thm:equivHorn}(ii).

\begin{Lemma}\label{lem:more_ineq}
Let $\lambda,\mu,\nu$ be partitions as in Theorem \ref{thm:equivHorn}. Then $\lambda,\mu,\nu$
satisfy condition (ii) of Theorem \ref{thm:equivHorn} if and only if for any triple $(I,J,K)$ such
that $\sigma_I\sigma_J\sigma_{K^{\vee}}\neq 0$ in the ordinary cohomology ring $H^*\Gr_d(\CC^r)$,
we have
$$\sum_{i\in I}\lambda_i+\sum_{j\in J}\mu_j\geq \sum_{k\in K}\nu_k.$$\end{Lemma}
\begin{proof}
If $c_{\itop(I),\itop(J)}^{\itop(K)}\neq 0,$ then $\sigma_I\sigma_J\sigma_{K^{\vee}}\neq 0$, so the
inequalities of the Lemma include those of Theorem~\ref{thm:equivHorn}(ii), which proves the ``if''
statement.  For the ``only if'' statement, we first recall a well-known fact (with proof, for
completeness):
\begin{Claim}
\label{claim:containment} If $\sigma_{\alpha}\sigma_{\beta}\sigma_{\gamma^{\vee}}\neq 0$ then there
exists $\widetilde{\gamma}$ such that $\widetilde{\gamma}\subseteq\gamma$ and
$c_{\alpha,\beta}^{\widetilde{\gamma}}\neq 0$.
\end{Claim}
\noindent \emph{Proof of Claim~\ref{claim:containment}:} We proceed by induction on
\[\Delta=|\alpha|+|\beta|+|\gamma^{\vee}|.\]
If $\Delta=d(r-d)$, then we can take ${\widetilde \gamma}=\gamma$.  If $\Delta<d(r-d),$ then the
product $\sigma_{\alpha}\sigma_{\beta}\sigma_{\gamma^{\vee}}$ does not lie in the top degree of
$H^*\Gr_d(\CC^r)$.  Thus $\sigma_{\alpha}\sigma_{\beta}\sigma_{\gamma^{\vee}}\neq 0$ implies
$$\sigma_{(1)}\sigma_{\alpha}\sigma_{\beta}\sigma_{\gamma^{\vee}}\neq 0.$$
In particular, we can choose $\gamma^{\uparrow}$ such that $\gamma^{\vee}\subseteq
\gamma^{\uparrow}$, $|\gamma^{\uparrow}|=|\gamma^{\vee}|+1$ and
$\sigma_{\alpha}\sigma_{\beta}\sigma_{\gamma^{\uparrow}}\neq 0$.  By induction, we can choose
$\widetilde\gamma^{\vee}$ such that $\gamma^{\uparrow}\subseteq\widetilde\gamma^{\vee}$ and
$c_{\alpha,\beta}^{\widetilde{\gamma}}\neq 0$.  But then
$\gamma^{\vee}\subseteq\gamma^{\uparrow}\subseteq\widetilde\gamma^{\vee}$, and hence
$\widetilde{\gamma}\subseteq\gamma$.  This proves the claim. \qed

Now, if $\sigma_I\sigma_J\sigma_{K^{\vee}}\neq 0$, then by the claim there exists $\widetilde K$
such that $\itop(\widetilde K)\subseteq \itop(K)$ and $c_{\itop(I),\itop(J)}^{\itop(\widetilde
K)}\neq 0$. Thus, the condition of Theorem~\ref{thm:equivHorn}(ii) implies
$$\sum_{i\in I}\lambda_i+\sum_{j\in J}\mu_j\geq \sum_{k\in \widetilde K}\nu_k,$$
and since $\itop(\widetilde K)\subseteq \itop(K)$, we have
\[
\sum_{k\in \widetilde K}\nu_k \geq \sum_{k\in K}\nu_k.
\]
Combining these two inequalities yields the inequality of the lemma.
\end{proof}

Lemma \ref{lem:more_ineq} allows us to replace the inequalities of Theorem \ref{thm:equivHorn}(ii)
by a larger set of inequalities.  That is, we instead use inequalities corresponding to
$(I,J,K)$ from the sets
$$S_d^r:=\{(I,J,K)\subseteq [r]^3\ |\
|I|=|J|=|K|=d\ \text{and}\ \sigma_I\sigma_J\sigma_{K^{\vee}}\neq 0 \text{ in } H^*\Gr_d(\CC^r)\},$$
for $d<r$.  This larger class of inequalities allows us to perform the induction.

We will first need a result from \cite{Fulton}. Let $I=\{i_1<i_2<\cdots< i_d\}$ be a subset of $[r]$ of cardinality $d$ and let $F$ be a subset of
$[d]$ of cardinality $x$.  Define $$I_F:=\{i_f\ | f\in F\}.$$  If $F$ is a subset of $[r-d]$ of cardinality $y,$ then define $$I_F^+:=I\cup (I^c)_F$$ where $I^c$ denotes the complement of $I$ in
$[r].$

\begin{Proposition}\label{prop:Fulton}(\cite[Proposition 1]{Fulton})
Let $(I,J,K)\in S_d^r.$
\begin{enumerate}
\item If $(F,G,H)\in S_x^d,$ then $(I_F,J_G,K_H)\in S_x^r$. \label{p.fulton1}

\item If $(F,G,H)\in S_y^{r-d},$ then $(I_F^+,J_G^+,K_H^+)\in S_{d+y}^r$. \label{p.fulton2}

\end{enumerate}\end{Proposition}

For any partitions $\lambda$ and $\alpha,$ define $\phi(\lambda,\alpha)$ to be the partition with
parts
$$\lambda_1,\ldots, \lambda_d,\alpha_1,\ldots,\alpha_d$$ arranged in weakly decreasing
order.

\begin{Lemma}\label{lem:split}
Let $\lambda,\mu,\nu$  and $\alpha,\beta,\gamma$ be partitions such that $C_{\lambda,\mu}^{\nu}\neq
0$ and $C_{\alpha,\beta}^{\gamma}\neq 0.$
Then $C_{\phi(\lambda,\alpha),\phi(\mu,\beta)}^{\phi(\nu,\gamma)}\neq 0.$
\end{Lemma}

\begin{proof}
Since $C_{\lambda,\mu}^{\nu}\neq 0$ and $C_{\alpha,\beta}^{\gamma}\neq 0,$ by Proposition~\ref{prop:technical}(\ref{propB}), there
exist partitions $\lambda^{\downarrow}\subseteq\lambda$ and $\beta^{\downarrow}\subseteq\beta$ such
that $c_{\lambda^{\downarrow},\mu}^{\nu}\neq 0$ and $c_{\alpha,\beta^{\downarrow}}^{\gamma}\neq 0.$
By Theorem \ref{thm:classicalHorn}(iii), there exist Hermitian matrices $A_1+B_1=C_1$ and
$A_2+B_2=C_2$ such that $A_i,B_i,C_i$ have eigenvalues $\lambda^\downarrow,\mu,\nu$ and $\alpha,\beta^\downarrow,\gamma$
respectively. Then
$$\left( \begin{array}{cc}
A_1 & 0 \\
0 & A_2  \end{array} \right)+ \left( \begin{array}{cc}
B_1 & 0 \\
0 & B_2  \end{array} \right)= \left( \begin{array}{cc}
C_1 & 0 \\
0 & C_2  \end{array} \right)$$ are Hermitian matrices with eigenvalues
$\phi(\lambda^{\downarrow},\alpha),\phi(\mu,\beta^{\downarrow}),\phi(\nu,\gamma)$ respectively.

Hence $c_{\phi(\lambda^{\downarrow},\alpha),\phi(\mu,\beta^{\downarrow})}^{\phi(\nu,\gamma)}\neq
0.$  We have
\[\phi(\lambda^{\downarrow},\alpha)\subseteq\phi(\lambda,\alpha)\subseteq \phi(\nu,\gamma) \mbox{\ and \ }
\phi(\mu,\beta^{\downarrow})\subseteq\phi(\mu,\beta)\subseteq \phi(\nu,\gamma).\]
Thus, by Proposition~\ref{prop:technical}(\ref{propA}), we conclude
$C_{\phi(\lambda,\alpha),\phi(\mu,\beta)}^{\phi(\nu,\gamma)}\neq 0.$\end{proof}

\begin{Remark}
The converse of Lemma \ref{lem:split} does not hold.  For example, we have 
$c_{(3),(2,1,1)}^{(3,2,1,1)}\neq 0,$ but $c_{(3),(1)}^{(2,1,1)}=c_{ \emptyset ,(2,1)}^{(3)}=0.$\qed
\end{Remark}

We will need the following:

\begin{Lemma}\label{lem:complement}
Let $\mu,\nu$ be partitions with at most $r$ parts such that $\mu\subseteq\nu$, and let $I,J$ be subsets
of $[r]$ of cardinality $d$.  If $\itop(I)\subseteq\itop(J)$, then $\mu_{I^c}\subseteq
\nu_{J^c}$.
\end{Lemma}
\begin{proof}
Let $I=(i_1<\cdots<i_d)$ and $J=(j_1<\cdots<j_d)$.  Since $\itop(I)\subseteq\itop(J)$, we have that $i_k\leq j_k$ for all $k\in[d]$.  This implies that if $$I^c=(i'_1<\cdots< i'_{r-d})\quad\text{and}\quad J^c=(j'_1<\cdots<j'_{r-d}),$$ then $i'_k\geq j'_k$ for all $k\in[r-d]$.  Now $\mu_{i'_k}\leq \nu_{j'_k}$, since  $\mu\subseteq\nu$.
\end{proof}

\noindent
\emph{Proof of Theorem \ref{thm:equivHorn}, (ii) $\Rightarrow$ (i):}
Let $p:=|\lambda|+|\mu|-|\nu|$. We prove the converse by a double induction on $p$ and $r$.
If $p=0$, then (ii) implies (i) by Theorem \ref{thm:classicalHorn}.  The second induction is on $r$.
In particular, if $r=1$, then $C_{\lambda,\mu}^{\nu}\neq0$ if and only if $\lambda_1+\mu_1\geq
\nu_1$.  These are the base cases of our induction.

Now assume $p>0$ and $r>1$. Suppose that $(\lambda,\mu,\nu)$ satisfies the inequalities given by
triples of subsets in $S_d^r$ for all $d<r$. In order to reach a contradiction, suppose
$C_{\lambda,\mu}^{\nu}=0$.

Since $p>0,$ we can assume that $|\lambda|\geq 1.$  Remove any box from $\lambda$, to obtain a
subpartition $\lambda^{\downarrow}\subseteq\lambda$ with $|\lambda|=|\lambda^{\downarrow}|+1$. By
Proposition~\ref{prop:technical}(\ref{propA}) and condition (\ref{eqn:EqHorncond}) of
Theorem~\ref{thm:equivHorn}, we know $C_{\lambda^{\downarrow},\mu}^{\nu}=0$.  By induction on $p$,
since $|\lambda^{\downarrow}|+|\mu|-|\nu|=p-1$, we can choose an inequality $(I,J,K)\in S_d^r$
satisfied by $(\lambda,\mu,\nu)$ but not by $(\lambda^{\downarrow},\mu,\nu)$.  It is therefore true
that
\begin{equation}
\label{eqn:actuallyeq}
\sum_{i\in I}\lambda_i+\sum_{j\in J}\mu_j= \sum_{k\in K}\nu_k.\quad
\end{equation}
Let $\lambda_I:=(\lambda_{i_1}\geq\cdots\geq\lambda_{i_d})$ and similarly define $\mu_J$ and
$\nu_K.$  Note that equation \eqref{eqn:actuallyeq} is the statement $|\lambda_I|+|\mu_J|=|\nu_K|$.
By assumption, $(\lambda,\mu,\nu)$ satisfies $(I,J,K)$. If $(F,G,H)\in S_{x}^d$ (for any $x<d$)
then by Proposition \ref{prop:Fulton}(\ref{p.fulton1}), $(\lambda,\mu,\nu)$ also satisfies $(I_F,I_G,K_H)\in
S_{x}^r$. This is the same as saying that $(\lambda_I,\mu_J,\nu_K)$ satisfies the inequalities
$(F,G,H)\in S_x^d$ for all $x<d$.  Thus, by Theorem \ref{thm:classicalHorn},
$c_{\lambda_I,\mu_J}^{\nu_K}\neq 0$.

Consider the partitions $\lambda_{I^c},\mu_{J^c},\nu_{K^c}$, where $I^c,J^c,K^c$ are the complements
of $I,J,K$ in $[r]$.  By Claim~\ref{claim:containment}, $\itop(I)\subseteq\itop(K)$, so by Lemma~\ref{lem:complement}, $\lambda_{I^c}$ is a
subpartition of $\nu_{K^c}$. Similarly, we have that $\mu_{J^c}$ is a subpartition of $\nu_{K^c}$.
Now, again using our assumption that $(\lambda,\mu,\nu)$ satisfies the inequalities $(I,J,K)\in
S_{d}^r$, it follows from Proposition \ref{prop:Fulton}(\ref{p.fulton2}) and equation \eqref{eqn:actuallyeq} that
$(\lambda_{I^c},\mu_{J^c},\nu_{K^c})$ satisfies the inequalities $(F,G,H)\in S_y^{r-d}$ for $y<r-d$.
Since $r-d<r$, we have by induction that $C_{\lambda_{I^c},\mu_{J^c}}^{\nu_{K^c}}\neq 0$.

Since $c_{\lambda_I,\mu_J}^{\nu_K}\neq 0$ and $C_{\lambda_{I^c},\mu_{J^c}}^{\nu_{K^c}}\neq 0$,
Lemma \ref{lem:split} implies that $C_{\lambda,\mu}^{\nu}\neq 0$, which contradicts our original assumption. This completes the proof.\qed

\section{Proof of Proposition~\ref{prop:technical}}
\subsection{A Littlewood-Richardson rule via edge-labeled tableaux}

 There are several combinatorial rules for computing the equivariant structure
constants $C_{\lambda,\mu}^\nu$; see, e.g., \cite{Molev.Sagan:99} and \cite{Knutson.Tao:03} for
some early ones (the latter being the first one to manifest the ``Graham-positivity'' of the polynomials
$C_{\lambda,\mu}^{\nu}$). However, in order to prove Proposition~\ref{prop:technical} we use a more recent rule of
\cite{Thomas.Yong}.

Consider Young diagrams $\lambda,\mu,\nu$ inside $\Lambda$, with $\lambda,\mu\subseteq \nu$.  An \define{equivariant Young tableau} of shape $\nu/\lambda$ and content $\mu$ is a filling of the boxes of the skew shape $\nu/\lambda$, and labeling of some of the edges, by integers $1,2,\ldots,|\mu|$, where $1$ appears $\mu_1$ times, $2$ appears $\mu_2$ times, etc.  The edges that can be labeled are the horizontal edges of boxes in $\nu/\lambda$, as well as edges along the southern border of $\lambda$; several examples are given below.
The tableau is \define{semistandard} if the box labels weakly increase along rows (left to right), and all labels strictly increase down columns.  A single edge may be labeled by a \emph{set} of integers, without repeats; the smallest of them must be strictly greater than the label of the box above, and the largest must be strictly less than the label of the box below.

\begin{Example} Below is an equivariant semistandard Young tableau on $(4,2,2)/(2,1)$.
\[
\begin{picture}(100,60)
\put(20,40){$\tableau{{\ }&{\ }&{1 }&{1 }\\{ \ }&{1}\\{2}&{3}}$}
\put(60,37){$2,3$}
\put(83,37){$2$}
\end{picture}
\]
The content of this tableau is $(3,3,2)$.
\qed
\end{Example}

Let $\eqssyt(\nu/\lambda)$ be the set of all equivariant semistandard Young tableaux of shape $\nu/\lambda$.
A tableau $T\in \eqssyt(\nu/\lambda)$ is \define{lattice} if, for every column $c$ and every label $\ell$, we have:
\[
  (\#\text{ $\ell$'s weakly right of column }c ) \geq (\#\text{$(\ell+1)$'s weakly right of column }c ).
\]
The lattice condition can also be phrased in terms of the \define{column reading word} $w(T)$, which is obtained by reading the columns from top to bottom, starting with the rightmost column and moving to the left.  (In the example above, $w(T) = 1\;2\;1\;2\;3\;1\;3\;2$.)  The tableau $T$ is lattice if and only if $w(T)$ is a lattice word---that is, for each $\ell$ and each $p$, among the first $p$ letters of $w(T)$, the number of $\ell$'s that appear is at least the number of $\ell+1$'s.

Given a tableau $T\in \eqssyt(\nu/\lambda)$, a (box or edge) label $\ell$ is \define{too high} if it appears weakly above the upper edge of a box in row $\ell$.  In the above example, all edge labels are too high.
(When there are no edge labels, the semistandard and lattice conditions imply no box label is too high, but in general the three conditions are independent.)

Each box in the $r \times (n-r)$ rectangle $\Lambda$ has a \define{distance} from the lower-left box: Using matrix coordinates for a box $\xbox=(i,j)$, we define
\[\dist(\xbox) = r+j-i.\]

Now suppose an edge label $\ell$ lies on the bottom edge of a box $\xbox$ in row $i$.  Let $\rho_\ell(\xbox)$ be the number of times $\ell$ appears as a (box or edge) label strictly to the right of $\xbox$.  We define
\begin{equation}
\label{eqn:apfactor}
  \factor(\ell,\xbox) = t_{\dist(\xbox)} - t_{\dist(\xbox)+i-\ell+1+\rho_\ell(\xbox)}.
\end{equation}
When the edge label is not too high, this is always of the form $t_p-t_q$, for $p<q$.  (In particular, it is nonzero.)  Finally, we define\footnote{In \cite{Thomas.Yong}, this is called the ``{\it a priori} weight'', to distinguish it from a weight arising from a sliding algorithm; hence the prefix ``$\mathtt{ap}$''.} the \define{weight} of $T\in \eqssyt$ by
\begin{equation}
\label{eqn:apwt}
  \wt(T) = \prod \factor(\ell,\xbox),
\end{equation}
the product being over all edge labels $\ell$.

We can now state a combinatorial rule for equivariant Schubert calculus.

\begin{Theorem}[{\cite[Theorem~3.1]{Thomas.Yong}}] \label{thm:ty}
We have $C_{\lambda,\mu}^\nu = \sum_T \wt(T)$, where the sum is over all $T\in \eqssyt(\nu/\lambda)$
of content $\mu$ that are lattice and have no label which is too high.
\end{Theorem}

This is a nonnegative rule: from the definition of the weights, complete cancellation is impossible,
so the existence of one such $T$ means that $C_{\lambda,\mu}^\nu$ is nonzero.  In particular:

\begin{Corollary}
\label{cor:nonzeroness} The coefficient $C_{\lambda,\mu}^\nu$ is nonzero if and only if there
exists a tableau $T\in \eqssyt(\nu/\lambda)$ of content $\mu$ which is lattice and has no label
which is too high.
\end{Corollary}

\begin{Example}
Consider the following lattice and semistandard tableau:
\[
\begin{picture}(100,60)
\put(20,40){$\tableau{{\ }&{\ }&{\ }&{\ }\\{ \ }&{1}\\{2}&{2}}$}
\put(65,37){$1$}
\put(85,37){$1$}
\put(26,-2){$3$}
\end{picture}
\]
The associated $\wt$ is $(t_1-t_2)(t_4-t_6)(t_5-t_6)$; hence $C_{(4,1),(3,2,1)}^{(4,2,2)}\neq 0$.\qed
\end{Example}

\subsection{Proof of Proposition~\ref{prop:technical}}
Our arguments for (\ref{propA}) and (\ref{propB}) are combinatorial, and both are based on Corollary~\ref{cor:nonzeroness}.

\noindent
\emph{Proof of (\ref{propA}):}  Let $T$ be a witnessing tableau for $C_{\lambda,\mu}^{\nu}\neq 0$. That is, $T$ is an (equivariant)
semistandard tableau of shape $\nu/\lambda$ that is lattice, has content $\mu$,
and has no label that is too high.
If $\mu=\nu$ then
the desired assertion is trivial. Otherwise, by induction we quickly reduce to the case that $\mu^{\uparrow}/\mu$ is a single
box. Suppose this additional box is a corner added to row $\ell$ of the shape of $\mu$. Our goal is to construct
$T^{\uparrow}$ by adding a single edge label $\ell$ to $T$ so that $T^{\uparrow}$ witnesses $C_{\lambda,\mu^{\uparrow}}^{\nu}\neq 0$.

\noindent
{\sf Procedure to obtain $T^{\uparrow}$:} Find the leftmost column $c$ that
\begin{itemize}
\item does not already have $\ell$ in the same column; and
\item placing $\ell$ in that column as an edge label does not make that new $\ell$ \emph{too high}.
\end{itemize}
Place $\ell$ in column $c$ as an edge label. (This placement is uniquely determined.)

First we need to establish:
\begin{Claim}\label{cexists}
The column $c$ exists.
\end{Claim}
\begin{proof} To reach a contradiction, suppose otherwise. First assume there is a
column $d$
of $T$ that does not have $\ell$ in it. We can take this column to be leftmost among all choices.

By our assumption for contradiction, we could not
insert $\ell$ into column $d$ because doing so would make
it too high. Thus the edge we would put $\ell$ into (as forced by semistandardness) is the upper edge of the box in row $\ell$, or higher. If there is a box label $m$ in the box of row $\ell$ in that column, then
$m>\ell$ (by assumption). But then $m$ was too high in $T$, a contradiction. Hence it must be the case 
that the $d$th column of $\nu$ has at most $\ell-1$ boxes.
Since there were $\ell$'s in each of the columns to the left, we conclude that
the corner box $\mu^{\uparrow}/\mu$ must be in column $d$ or to the right.  
This means the $d$th column of $\mu^{\uparrow}$ has at least $\ell$ boxes, which contradicts the assumption $\mu^{\uparrow} \subseteq \nu$.
Finally, if column $d$ does not exist, i.e., every column
of $T$ has an $\ell$ in it, then $\mu_{\ell}=\nu_{\ell}$ and thus
$\mu^{\uparrow}\not\subseteq\nu$, a contradiction.
\end{proof}

Since we are adding $\ell$ to an edge, the assumptions imply
that $T^{\uparrow}$ is semistandard (as the horizontal semistandard condition is vacuous here). 
\begin{Claim}
$T^{\uparrow}$ is lattice.
\end{Claim}
\begin{proof} Suppose $T^{\uparrow}$ is not lattice. So there is a column $d$ with strictly
more $\ell$'s than $(\ell-1)$'s in the region $R$ consisting of columns weakly to the right. Notice
that since $T$ is lattice and we put an additional $\ell$ in column $c$, then column $d$ must be weakly left of column $c$.

Before inserting the $\ell$, the region $R$ had an equal number of $\ell$'s and $(\ell-1)$'s. Since we could
put an $\ell$ into column $c$ and not be too high, we could put an edge label in each of the columns strictly
left of column $d$, unless they all had $\ell$'s in them. However, in that case, since $T$ is lattice, those columns
must each also contain $\ell-1$. Thus,
$\mu_{\ell-1}=\mu_{\ell}$. Hence we could not add a corner in row $\ell$ to obtain $\mu^{\uparrow}$,
a contradiction.
\end{proof}

This completes the proof of (A). \qed

\begin{Example}
We illustrate the procedure below:
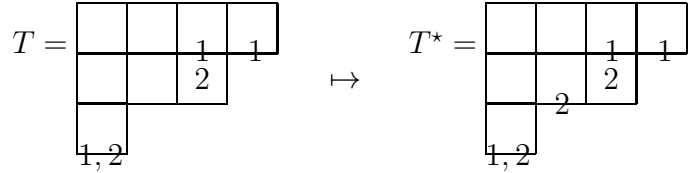
\[
\begin{picture}(300,60)
\put(20,40){$T=\tableau{{\ }&{\ }&{\ }&{\ }\\{ \ }&{\ }&{2}\\{\ }}$}
\put(89,37){$1$}
\put(109,37){$1$}
\put(45,-2){$1,2$}
\put(140,27){$\mapsto$}
\put(170,40){$T^{\star}=\tableau{{\ }&{\ }&{\ }&{\ }\\{ \ }&{\ }&{2}\\{\ }}$}
\put(244,37){$1$}
\put(264,37){$1$}
\put(225,17){$2$}
\put(199,-2){$1,2$}
\end{picture}
\]
Above, $T$ witnesses $C_{(4,2,1),(3,2)}^{(4,3,1)} \neq 0$ and $T^{\star}$ witnesses $C_{(4,2,1),(3,3)}^{(4,3,1)} \neq 0$.\qed
\end{Example}

\medskip
\noindent\emph{Proof of (\ref{propB}):} As in the proof of (\ref{propA}), let $T$ be a witnessing tableau for $C_{\lambda,\mu}^{\nu}\neq 0$.  We will modify $T$ to obtain $T^{\star}$ that witnesses $C_{\lambda,\mu^{\star}}^{\nu}\neq 0$, where $\mu^{\star}\subset \mu$ and $|\mu/\mu^{\star}|=1$. ($T^{\star}$ will have one fewer edge label than $T$.)  The claim (\ref{propB}) follows by using this procedure to obtain a sequence of tableaux, each with one fewer edge label, until there are no edge labels.

\noindent {\sf Procedure to obtain $T^{\star}$:} We introduce some temporary notation.  For a word
$w$, a position $p$, and a letter $\ell$, let $N(w,p,\ell)$ be the number of occurrences of $\ell$
among the first $p$ letters of $w$.  Thus the lattice condition is that $N(w,p,\ell)\geq
N(w,p,\ell+1)$ for all $\ell$ and all $p$.

Now consider the (top to bottom, right to left) column reading word $w(T)$, and find the last
occurrence of an edge label; say this label is $\ell$, occurring in position $p^{(1)}$.  Remove
this label to obtain a new word $w^{(1)}$, and continue reading along the word, letter by letter. At
any position $q\geq p^{(1)}$, and for any letter $k\neq \ell$, we have \[N(w^{(1)},q,k) =
N(w(T),q,k) \mbox{\ \ and \ $N(w^{(1)},q,\ell)=N(w(T),q,\ell)-1$.}\]
If at some position $p^{(2)}\geq
p^{(1)}$, the lattice condition is violated in $w^{(1)}$, it must be because a letter $\ell+1$
appeared, causing
\[N(w^{(1)},p^{(2)},\ell+1)>N(w^{(1)},p^{(2)},\ell).\]
Fix this violation by
replacing this problematic $\ell+1$ with $\ell$, and call the resulting word $w^{(2)}$.  Note that
$w^{(2)}$ is lattice up to position $p^{(2)}$; moreover, for any $q \geq p^{(2)}$ and any $k\neq
\ell+1$, we have
\[
N(w^{(2)},q,k)=N(w(T),q,k) \mbox{\ \ and  \ $N(w^{(2)},q,\ell+1) = N(w(T),q,\ell+1)-1$.}
\]
Continue in this way until the end of the word is reached, and call the result $w^\star$.

By construction, $w^\star$ is a lattice word.  Furthermore, after the $\ell$ removed in the first
step, the only letters in which $w^\star$ differs from $w(T)$ correspond to box labels.  So there
is no ambiguity in how to place these entries to create a tableau $T^*$ of the same shape as $T$.

Let $\mu^\star$ be the content of $T^\star$.  The argument is completed by Lemma~\ref{l.witness} below.\qed

\begin{Lemma}\label{l.witness}
The tableau $T^\star$ witnesses $C_{\lambda,\mu^\star}^\nu \neq 0$.  In particular,
\begin{enumerate}[(i)]
\item $T^\star$ is lattice.

\item $T^\star$ is semistandard.

\item No label of $T^{\star}$ is too high.%\qed
\end{enumerate}
\end{Lemma}

Before proving the lemma, we illustrate the procedure with an example.

\begin{Example}\label{Ex1}
Consider the tableau $T$:
\[
\begin{picture}(300,60)
\put(20,40){$T=\tableau{{\ }&{ \ }&{\ }&{\ }&{\ }&{1 }\\{ \ }&{\ }&{2}&{2}\\{3}&{3}}$}
\put(108,37){$1$}
\put(180,40){$T^{\star}=\tableau{{\ }&{ \ }&{\ }&{\ }&{\ }&{1 }\\{ \ }&{\ }&{1}&{2}\\{2}&{3}}$}
\end{picture}
\]
The reading word is $w(T) = 1\;{\underline 1}\;2\;2\;3\;3$, with the edge label underlined.  Removing this letter, we have a violation of the lattice condition in the third position, which is fixed by:
\[
 w^{(1)} = 1\;2\;2 \ldots \qquad \rightsquigarrow \qquad w^{(2)} = 1\;2\;1\ldots.
\]
Continuing, another violation is at the fifth (and last) position, and we fix it as before:
\[
 w^{(2)} = 1\;2\;1\;3\;3 \qquad \rightsquigarrow \qquad w^{(3)} = 1\;2\;1\;3\;2.
\]
The corresponding tableau $T^\star$ is shown above.
\qed
\end{Example}

For use in the proof of Lemma~\ref{l.witness}, it will be convenient to let $T^{(i)}$ denote the tableau corresponding to an intermediate word $w^{(i)}$.  This may not be a lattice tableau, but from the construction $T^{(i)}$ does satisfy the lattice condition with respect to labels $\ell+i$ and $\ell+i+1$.

\begin{Example}\label{Ex2}
The intermediate tableaux for Example~\ref{Ex1} are shown below.

\begin{center}
\begin{picture}(300,65)(0,-70)
\put(10,-30){$T^{(1)}=\tableau{{\ }&{ \ }&{\ }&{\ }&{\ }&{1 }\\{ \ }&{\ }&{2}&{2}\\{3}&{3}}$}

\put(180,-30){$T^{(2)}=\tableau{{\ }&{ \ }&{\ }&{\ }&{\ }&{1 }\\{ \ }&{\ }&{1}&{2}\\{3}&{3}}$}

%\put(80,-30){$T^{(3)}=T^{\star}=\tableau{{\ }&{ \ }&{\ }&{\ }&{\ }&{1 }\\{ \ }&{\ }&{1}&{2}\\{2}&{3}}$}

\end{picture}
\end{center}
and fixing $T^{(2)}$ gives $T^{(3)}=T^{\star}$, which was already given above.
\qed
\end{Example}

\begin{proof}[Proof of Lemma~\ref{l.witness}]
As was already observed, (i) is true by construction.  The claim (iii) is also easy to verify: We are given that no label of $T$ is too high, and removing the initial edge label does not change this.  Moreover, each step of our procedure only changes a box label to a label that is one smaller.  If such a box label was not too high, replacing it by a smaller label will not change this either.

It remains to establish (ii), which we will do by induction on the steps of the procedure.  To obtain $T^{(i+1)}$ from $T^{(i)}$, a box label $\ell+i$ is replaced by $\ell+i-1$.  Let $c^{(i)}$ be the column where this replacement occurs (and let $c^{(0)}$ be the column of the edge label $\ell$ that was initially removed from $T$).

\begin{Claim}
\label{claim:nosuchlabel}
Suppose $i\geq 1$. In column $c^{(i)}$ of $T^{(i)}$ there is an $\ell+i$ (by assumption), but there is no $\ell+i-1$.
\end{Claim}
\begin{proof}
Were this not the case, the lattice condition would be violated by a label $\ell+i$ occurring earlier in the reading word $w^{(i)}$, but by construction we chose the first violation.
\end{proof}

Now assume $T^{(i)}$ is semistandard.  In the following argument, it will help to refer to a diagram illustrating the replacement taking $T^{(i)}$ to $T^{(i+1)}$, locally:
\[T^{(i)}=\begin{tabular}{ |l | c | r | }
\hline
  $\sf a$ & $\sf b$ & $\sf c$ \\
  \hline
  $\sf d$ & $\ell+i$ & $\sf e$ \\
  \hline
  $\sf f$ & $\sf g$ & $\sf h$ \\
\hline
\end{tabular}
\to
\begin{tabular}{ |l | c | r | }
\hline
  $\sf a$ & $\sf b$ & $\sf c$ \\
  \hline
  $\sf d$ & $\ell+i-1$ & $\sf e$ \\
  \hline
  $\sf f$ & $\sf g$ & $\sf h$ \\
\hline
\end{tabular}
=T^{(i+1)}.\]
(If the second column in the above local diagram is $c^{(0)}$
there could be edge labels above the $\sf b$ or in the columns to the right.  However, these labels do not affect the argument.)

\begin{Claim}\label{ssyt}
$T^{(i+1)}$ is semistandard.
\end{Claim}
\begin{proof}
Clearly we have $\ell+i-1\leq {\sf e}$ and $\ell+i-1<{\sf g}$. That ${\sf b}<\ell+i-1$ is clear from the
semistandardness of $T^{(i)}$ combined with Claim~\ref{claim:nosuchlabel}.  It remains to show ${\sf d}\leq \ell+i-1$.

We have ${\sf d}\leq \ell+i$, so in order to reach a contradiction,
let us assume ${\sf d}=\ell+i$.  There are two cases.
First, if ${\sf a}=\ell+i-1$,
then by the semistandardness of $T^{(i)}$ we have ${\sf b}=\ell+i-1$, contradicting Claim~\ref{claim:nosuchlabel}.
Second, if ${\sf a}<\ell+i-1$ (or if $a$ does not exist), then it follows from Claim~\ref{claim:nosuchlabel} that in $T^{(i-1)}$, the label ${\sf d}=\ell+i$
witnesses that $T^{(i-1)}$ is not lattice with respect to $\ell+i-1$ and $\ell+i$, a contradiction since we know $T^{(i-1)}$ is lattice with respect to these labels.  Hence ${\sf d}<\ell+i$, and $T^{(i+1)}$ is semistandard.
\end{proof}

This completes the proof of the lemma.
\end{proof}

%%%%%%%%%%%%%%%%%%%%%%%%%%%%
\section*{Acknowledgements}
%%%%%%%%%%%%%%%%%%%%%%%%%%%%
The authors are grateful to the organizers of the MSJ-SI 2012 Schubert calculus conference in Osaka, Japan, where this collaboration was initiated.
AY thanks Hugh Thomas for their earlier work on \cite{Thomas.Yong} and the insights he shared during that process.
%DA was partially supported by NSF Grant DMS-0902967.
AY was partially supported by an NSF grants DMS-0901331 and DMS-1201595, a Beckman Fellowship from UIUC's Center for Advanced Study, and support from the Helen Corley Petit endowment, also at UIUC.  ER was partially supported by the Natural Sciences and Engineering Research Council of Canada.

%%%%%%%%%%%%%%%%%%%%%%%%%%%%%%%%%%%%%%%%%

%%%%%%%%%%%%%%%%%%%%%%%%%%%%%%%%%%%%%%%%%

%%%%%%%%%%%%%%%%%%%%%%%%%%%%%%%%%%%%%%%%%

\begin{thebibliography}{9999999}
%%%%%%%%%%%%%%%%%%%%%%%%%%%%%%%%%%%%%%%%%

\bibitem[An07]{Anderson} D.~Anderson, \emph{Positivity in the cohomology of flag bundles (after Graham)}, preprint, 2008. \textsf{arXiv:0711.0983}

\bibitem[Be06]{Belkale:06} P.~Belkale, \emph{Geometric proofs of Horn and saturation conjectures}, J. Alg. Geom. {\bf 15} (2006), no. 1, 133--173.

\bibitem[Be08]{Belkale:08} \bysame, \emph{Quantum generalization of the Horn conjecture}, J. Amer. Math. Soc. {\bf 21} (2008), no.
2, 365--408.

\bibitem[BeKu06]{BK:06} P.~Belkale and S.~Kumar, \emph{Eigenvalue problem
and a new product in cohomology of flag varieties}, Invent.~Math. {\bf 166} (2006), no. 1, 185--228.

\bibitem[BeKu10]{BK:10} \bysame, \emph{Eigencone, saturation
and Horn problems for symplectic and odd orthogonal groups}, J.~Algebraic Geom. {\bf 19} (2010), 199--242.

\bibitem[Bu02]{Buch:02} A.~Buch, \emph{A Littlewood-Richardson rule for the K-theory of Grassmannians}, Acta Math. {\bf 189} (2002), no. 1, 37--78.

\bibitem[Bu06]{Buch:06} \bysame, \emph{Eigenvalues of Hermitian matrices with positive sum of bounded rank},
Linear Algebra Appl. {\bf 418} (2006), no. 2-3, 480--488.

\bibitem[Ch06]{Chindris} C.~Chindris, \emph{Eigenvalues of Hermitian matrices and cones arising from quivers}, Inter.~Math.~Res.~Notices, 2006, Art. ID 59457, 27 pages.

\bibitem[DeWe00]{Derksen.Weyman} H.~Derksen, J.~Weyman, \emph{Semi-invariants of quivers and saturation for Littlewood-Richardson coefficients}, J.~Amer.~Math.~Soc. {\bf 13} (2000), no.~3, 456--479.

\bibitem[Fr00]{Friedland} S.~Friedland, \emph{Finite and infinite dimensional generalizations of Klyachko's theorem}, Linear Algebra Appl., {\bf 319} (2000), 3--22.

\bibitem[Fu00a]{Fulton} W.~Fulton, \emph{Eigenvalues of majorized Hermitian matrices and Littlewood-Richardson coefficients},
Linear Algebra Appl. {\bf 319} (2000), no. 1--3, 23--36.

\bibitem[Fu00b]{Fultona} \bysame, \emph{Eigenvalues, invariant factors, highest weights, and Schubert calculus},
Bull. Amer. Math. Soc. (N.S.) {\bf 37} (2000), no. 3, 209--249 (electronic).

\bibitem[Fu07]{Fultoneq} \bysame, \emph{Equivariant cohomology in algebraic geometry},
  lectures at Columbia University, notes by D.\ Anderson, 2007.
  \textsf{http:/$\!$/www.math.washington.edu/$\sim$dandersn/eilenberg}

\bibitem[Ho62]{Horn} A.~Horn, \emph{Eigenvalues of sums of Hermitian matrices}, Pacific J.~Math.,
{\bf 12} (1962), 225--241.

\bibitem[KaMi08]{Kapovich.Millson} M.~Kapovich and J.~J.~Millson, \emph{A path model for geodesics in Euclidian buildings and its applications to representation theory}, Groups Geom. Dyn. {\bf 2} (2008), no.~3, 405--480.

\bibitem[Kl98]{Klyachko} A.~A.~Klyachko, \emph{Stable vector bundles and Hermitian operators},
Selecta Math. (N.S.) {\bf 4} (1998), 419--445.

\bibitem[KnTa99]{Knutson.Tao:99} A.~Knutson and T.~Tao, \emph{The honeycomb model of $GL_n({\mathbb C})$ tensor products I: proof of the saturation conjecture}, J.~Amer.~Math.~Soc. {\bf 12} (1999), 1055--1090.

\bibitem[KnTa03]{Knutson.Tao:03} \bysame, \emph{Puzzles and (equivariant) cohomology of Grassmannians},  Duke Math. J. {\bf 119} (2003), no. 2, 221--260.

\bibitem[Ku10]{Kumar:ICM} S.~Kumar, \emph{Tensor product decomposition}, International Congress of Mathematicians, Hyderabad, India, 2010.

\bibitem[MoSa99]{Molev.Sagan:99} A.~Molev and B.~Sagan, \emph{A Littlewood-Richardson rule for factorial Schur functions}, Trans. Amer. Math. Soc. {\bf 351} (1999), no. 11, 4429--4443.

\bibitem[PuSo09]{Purbhoo.Sottile} K.~Purbhoo and F.~Sottile,
\emph{The recursive nature of cominuscule Schubert calculus}, Adv.~Math., {\bf 217}(2008), 1962--2004.

\bibitem[Res10]{Ressayre} N.~Ressayre, \emph{Geometric invariant theory and the generalized eigenvalue problem}, Invent.~Math. {\bf 180} (2010), no. 2, 389--441.

\bibitem[Sa12]{Sam} S.~Sam, \emph{Symmetric quivers, invariant theory, and saturation theorems for the classical groups}, Adv. Math. 229 (2012), no. 2, 1104--1135.

\bibitem[ThYo10]{Thomas.Yong:direct} H.~Thomas and A.~Yong, \emph{The direct sum map on Grassmannians and jeu de taquin for increasing tableaux},
Int.~Math.~Res.~Notices (2011) Vol. 2011, 2766--2793.

\bibitem[ThYo12]{Thomas.Yong} \bysame, \emph{Equivariant Schubert calculus and jeu de taquin}, preprint, 2012. \textsf{arXiv:1207.3209}

\bibitem[To94]{Totaro} B.~Totaro, \emph{Tensor products of semistables are semistable}, in Geometry and Analysis on Complex Manifolds, Festschrift for Professor S.~Kobayashi's 60th Birthday, ed. T.~Noguchi, J.~Noguchi, and T.~Ochiai, World Scientific Publ.~Co., Singapore, 1994,  242--250.

%%%%%%%%%%%%%%%%%%%%%%%%%%%%%%%%%%%%%%%%%
\end{thebibliography}
\end{document}